\documentclass[a4paper,11pt]{article}
\usepackage[T1]{fontenc}
\usepackage{amsfonts}
\usepackage{verbatim}
\usepackage{amsmath, amsthm, amssymb}

\author{  Michael~Grabchak\footnote{Department of Mathematics and Statistics, University of North Carolina at Charlotte. E-mail:
   mgrabcha@uncc.edu}
 }

\title{Limit Theorems For Sequences of Tempered Stable and Related Distributions}

\newtheorem{proposition}{Proposition}[section]
\newtheorem{theorem}[proposition]{Theorem}
\newtheorem{definition}[proposition]{Definition}
\newtheorem{corollary}[proposition]{Corollary}
\newtheorem{lemma}[proposition]{Lemma}
\newtheorem{remark}[proposition]{Remark}

\numberwithin{equation}{section}

\newcommand{\rd}{\mathrm d}

\newcommand{\ts}{TS^p_\alpha}
\newcommand{\ets}{ETS^p_\alpha}
\newcommand{\tr}{\mathrm{tr}}

\newcommand{\conv}{\stackrel{v}{\rightarrow}}
\newcommand{\conw}{\stackrel{w}{\rightarrow}}

\begin{document}
\maketitle

\begin{abstract}
In this paper we define the closure under weak convergence of  the class of $p$-tempered $\alpha$-stable distributions. We give necessary and sufficient conditions for convergence of sequences in this class. Moreover, we show that any element in this class can be approximated by the distribution of a linear combination of elementary $p$-tempered $\alpha$-stable random variables.\\

\noindent\textbf{Key words:}  Tempered Stable Distributions; Limit Theorems; Infinite Divisibility; Thorin Class; Goldie-Steutel-Bondesson Class; Weak Convergence
\end{abstract}

\section{Introduction}

Tempered stable distributions were introduced in Rosi\'nski 2007 \cite{Rosinski:2007} as a class of models that are similar to stable distributions in some central region, but they have lighter tails. Such models have been used successfully in a variety of areas including physics, biostatistics, and mathematical finance (see the references in \cite{Grabchak:2011}). A discussion of why such models come up in applications is given in \cite{Grabchak:Samorodnitsky:2010}. 

In \cite{Grabchak:2011} the wider class of $p$-tempered $\alpha$-stable distributions ($\ts$), where $p>0$ and $\alpha<2$, was introduced. Rosi\'nski's class corresponds to the case when $p=1$ and $\alpha\in(0,2)$. Tempered infinitely divisible distributions defined in \cite{Bianchi:Rachev:Kim:Fabozzi:2011} are another subclass corresponding to the case when $p=2$ and $\alpha\in[0,2)$. If we allow the distributions to have a Gaussian part, then we would have the class $J_{\alpha,p}$ defined in \cite{Maejima:Nakahara:2009}. This, in turn, contains important subclasses including the Thorin class (when $p = 1$ and $\alpha = 0$), the Goldie-Steutel-Bondesson class (when $p = 1$ and $\alpha = -1$), the class of type M distributions (when $p = 2$ and $\alpha= 0$), and the class of type G distributions (when $p = 2$ and $\alpha = -1$). For more information about these classes see the references in \cite{Barndorff-Nielsen:Maejima:Sato:2006} and \cite{Aoyama:Maejima:Rosinski:2008}.

In this paper we discuss the possible weak limits of $\ts$ distributions. It turns out that this class is not closed under weak convergence. We introduce the class of extended $p$-tempered $\alpha$-stable distributions ($\ets$), which is the smallest class that contains $\ts$ and is closed under weak convergence. For $\alpha\le0$ it corresponds to the class $J_{\alpha,p}$, but for $\alpha\in(0,2)$ we have $J_{\alpha,p}\subsetneq \ets$.

We also show that every $d$-dimensional $\ets$ distribution can be approximated by the distribution of a linear combination of $1$-dimensional elementary $\ts$ random variables. Elementary $\ts$ distributions, more commonly called ``smoothly truncated L\'evy flights,'' form a well studied subclass of $\ts$. For $p=1$ and $\alpha\in(0,2)$ the problem of simulation from such distributions is considered in, e.g.\ \cite{Baeumer:Meerschaert:2010} and \cite{Kawai:Masuda:2011}. Combining this with our results allows for approximate simulation of $\ets$ random vectors in $d$-dimensions. We will consider this in a future work.

This paper is organized as follows. In Section \ref{sec: ets} we define the class $\ets$. We take a detour in Section \ref{sec: compactification} to discuss a compactification of $\mathbb R^d$. In Section \ref{sec: extended Rosinski measure} we define a measure on this compactification, which determines many properties of $\ets$ distributions. Then, in Section \ref{sec: Limits of Sequences of Extended Tempered Stable Distributions}, we give our main limit theorem. Finally, in Section \ref{sec: Closure}, we show that the class $\ets$ is, in fact, the closure of $\ts$ under weak convergence.

Before proceeding, recall that the characteristic function of an infinitely divisible distribution $\mu$ on $\mathbb R^d$ can be written as $\hat\mu(z) = \exp\{C_{\mu}(z)\}$ where
\begin{eqnarray}\label{eq: inf div char func}
C_{\mu}(z) = -\frac{1}{2}\langle z,Az\rangle + i\langle b,z\rangle + \int_{\mathbb R^d}\left(e^{i\langle z,x\rangle}-1-i\frac{\langle z,x\rangle}{1+|x|^2}\right)M(\rd x),
\end{eqnarray}
$A$ is a Gaussian covariance matrix, $M$ is a L\'evy measure, and $b\in\mathbb R^d$. The measure $\mu$ is uniquely identified by the L\'evy triplet $(A,M,b)$, and we write $\mu= ID(A,M,b)$. For details about infinitely divisible distributions, the reader is referred to \cite{Sato:1999}. We use the notation  $\mathbb S^{d-1}=\{x\in\mathbb R^d: |x|=1\}$, and we write $\mathfrak B(\mathbb R^d)$ to denote the Borel sets on $\mathbb R^d$.

\section{Extended Tempered Stable Distributions}\label{sec: ets}

We begin by recalling the definition of $p$-tempered $\alpha$-stable distributions given in \cite{Grabchak:2011}.

\begin{definition} \label{defn: TS}
Fix $\alpha<2$ and $p>0$. An infinitely divisible probability measure $\mu$ is called a $\mathbf p$\textbf{-tempered} $\mathbf \alpha$\textbf{-stable distribution} if it has no Gaussian part and its L\'evy measure is given by
\begin{eqnarray}\label{eq:m}
M(B) =\int_{\mathbb S^{d-1}}\int_0^\infty 1_B(ru)q(r^p,u)r^{-\alpha-1}\rd r\sigma(\rd u), \quad B\in\mathfrak B(\mathbb R^d),
\end{eqnarray}
where $\sigma$ is a finite Borel measure on $\mathbb S^{d-1}$ and $q:(0,\infty)\times\mathbb S^{d-1}\mapsto\mathbb (0,\infty)$ is a Borel function such that for all $u\in\mathbb S^{d-1}$ $q(\cdot,u)$ is completely monotone and
\begin{eqnarray}\label{eq: q goes to zero}
\lim_{r\rightarrow\infty}q(r,u)=0.
\end{eqnarray}
We denote the class of $p$-tempered $\alpha$-stable distributions by $\ts$. If, in addition,
\begin{eqnarray}
\lim_{r\downarrow0}q(r,u)=1
\end{eqnarray}
for every $u\in \mathbb S^{d-1}$ then $\mu$ is called a \textbf{proper} $\mathbf p$\textbf{-tempered} $\alpha$\textbf{-stable distribution}.
\end{definition}

In \cite{Grabchak:2011} it was shown that $M$ is the L\'evy measure of a $p$-tempered $\alpha$-stable distribution if and only if there is a Borel measure $R$ on $\mathbb R^d$ such that
\begin{eqnarray}\label{eq:levy m for TS}
M(B) = \int_{\mathbb R^d}\int_0^\infty 1_B(tx)t^{-1-\alpha}e^{-t^p}\rd t R(\rd x), & B\in\mathfrak{B}(\mathbb{R}^d).
\end{eqnarray}
We call $R$ the \textbf{Rosi\'nski measure} of the $p$-tempered $\alpha$-stable distribution. Moreover, $R$ is the Rosi\'nski measure of some $p$-tempered $\alpha$-stable distribution if and only if $R(\{0\})=0$ and
\begin{eqnarray}\label{eq: integ cond on R}
\int_{\mathbb R^d} \left(|x|^2\wedge|x|^\alpha\right)R(\rd x)<\infty & \mathrm{if} \ \alpha\in(0,2),\nonumber\\
\int_{\mathbb R^d} \left(|x|^2\wedge[1+\log^+|x|]\right)R(\rd x)<\infty & \mathrm{if}  \ \alpha=0,\\
\int_{\mathbb R^d} \left(|x|^2\wedge1\right)R(\rd x)<\infty & \mathrm{if} \ \alpha<0,\nonumber
\end{eqnarray}
where $\log^+|x|= 1_{|x|\ge1}\log|x|$. We now extend $\ts$ to the smallest class that contains it and is closed under weak convergence. To do this we must allow for a Gaussian part and remove the assumption that \eqref{eq: q goes to zero} holds.

\begin{definition}\label{defn: ETS}
Fix $\alpha<2$ and $p>0$. An infinitely divisible probability measure $\mu$ is called an \textbf{extended} $\mathbf p$\textbf{-tempered} $\mathbf \alpha$\textbf{-stable distribution} if  its L\'evy measure is given by \eqref{eq:m} where $\sigma$ is a finite Borel measure on $\mathbb S^{d-1}$ and $q:(0,\infty)\times\mathbb S^{d-1}\mapsto\mathbb (0,\infty)$ is a Borel function such that for all $u\in\mathbb S^{d-1}$ $q(\cdot,u)$ is completely monotone. We denote the class of extended $p$-tempered $\alpha$-stable distributions by $\ets$.
\end{definition}

\begin{remark}\label{remark: no stable for neg alpha}
When $\alpha\le0$ \eqref{eq: q goes to zero} is necessary to ensure that $M$ is a L\'evy measure. Thus, whenever $\alpha\le0$ an $\ets$ distribution is just a $\ts$ distribution with a Gaussian part.
\end{remark}

\begin{remark}\label{remark: closed under convolutions}
Since the sum of completely monotone functions is completely monotone, it follows that both the class $\ts$ and the class $\ets$ are closed under taking convolutions.
\end{remark}

By Bernstein's Theorem (see e.g.\  
\cite{Feller:1971}) the complete monotonicity of $q(\cdot,u)$ implies that there is a measurable family $\{Q_u\}_{u\in \mathbb S^{d-1}}$ of Borel measures on $[0,\infty)$ such that
\begin{eqnarray}\label{eq:Q(|)}
q(r,u) = \int_{[0,\infty)} e^{-rs}Q_u(\mathrm d s).
\end{eqnarray}
Letting
\begin{eqnarray}
q_1(r,u) = \int_{(0,\infty)} e^{-rs}Q_u(\mathrm d s)
\end{eqnarray}
gives
\begin{eqnarray}
q(r^p,u) = q_1(r^p,u)+Q_u(\{0\}).
\end{eqnarray}
Thus, the L\'evy measure of a distribution in $\ets$ is given by
\begin{eqnarray}
M(B) &=& \int_{\mathbb S^{d-1}}\int_0^\infty 1_B(ru)q_1(r^p,u)r^{-\alpha-1}\rd r\sigma(\rd u)\nonumber\\
&&\quad + \int_{\mathbb S^{d-1}}\int_0^\infty 1_B(ru)r^{-\alpha-1}\rd r Q_u(\{0\})\sigma(\rd u), \quad B\in\mathfrak B(\mathbb R^d).
\end{eqnarray}
Note that, by Remark \ref{remark: no stable for neg alpha}, when $\alpha\le0$ we have $Q_u(\{0\})=0$. Since $q_1(\cdot,u)$ is completely monotone for all $u\in\mathbb S^{d-1}$ and it satisfies \eqref{eq: q goes to zero}, 
$M$ is the sum of the L\'evy measure of a $p$-tempered $\alpha$-stable distribution and (when $\alpha\in(0,2)$) an $\alpha$-stable distribution with spectral measure $Q_u(\{0\})\sigma(\rd u)$. 
If $R$ is the Rosi\'nski measure of the $p$-tempered $\alpha$-stable part then
\begin{eqnarray}\label{eq: Levy measure of ets in terms of R}
M(B) &=& \int_{\mathbb R^d}\int_0^\infty 1_B(rx) r^{-1-\alpha}e^{-r^p}\rd r R(\rd x) \nonumber\\
&& \quad + \int_{\mathbb S^{d-1}}\int_0^\infty 1_B(r u) r^{-1-\alpha}\rd r Q_u(\{0\})\sigma(\rd u), \quad B\in\mathfrak B(\mathbb R^d).\label{eq: levy of ETS}
\end{eqnarray}
\begin{remark}
This implies that a distribution is in $\ets$ if and only if it can be written as the convolution of a Gaussian distribution, an element of $\ts$, and (when $\alpha\in(0,2)$) an $\alpha$-stable distribution. 
\end{remark}

Note that $M$ is defined in terms of two measures $R(\rd x)$ and $Q_u(\{0\})\sigma(\rd u)$. To make it easier to work with we combine these into one measure, which we will define on a particular compactification of $\mathbb R^d$.

\section{A Compactification of $\mathbb R^d$}\label{sec: compactification}

In this section we develop a compactification of $\mathbb R^d$ with a sphere at infinity. Vague convergence of Radon measures on this space will be fundamental to our main results, therefore we give a detailed discussion. 

Let $\mathbb R^d_0=\mathbb R^d\setminus\{0\}$ and note that for $x\in\mathbb R^d_0$ we have $x = |x|\frac{x}{|x|}$. Thus we can uniquely identify every element of $\mathbb R^d_0$ with an element of $(0,\infty)\times\mathbb S^{d-1}$. Let $\bar{\mathbb R}_0^d=(0,\infty]\times \mathbb S^{d-1}$ and $\bar{\mathbb R}^d = \bar{\mathbb R}_0^d\cup\{0\}$. For simplicity of notation define $\mathbb I^{d-1} = \{\infty\}\times\mathbb S^{d-1}$ and $\infty u = (\infty,u)$. We introduce the functions $\xi: \mathbb{\bar R}^d \mapsto \mathbb S^{d-1}\cup\{0\}$ and $\vartheta: \mathbb{\bar R}^d \mapsto [0,\infty]$ as follows. Let $\xi(0) = \vartheta(0)=0$. If $x\in \mathbb {\bar R}^d_0$ then $x=(r,u)$ and we define $\xi(x)=u$ and $\vartheta(x)=r$. For simplicity, we sometimes write $|x|:=\vartheta(x)$, and when $x\in\mathbb I^{d-1}$ we take $|x|^{-1}=1/|x|=0$.

Let $\stackrel{\bar{\mathbb R}_+}{\rightarrow}$ and $\stackrel{\mathbb R^d}{\rightarrow}$ denote, respectively, the usual convergence on $[0,\infty]$ and on $\mathbb R^d$. If $x,x_1,x_2, \dots \in \bar{\mathbb R}^d_0$, we will write 
$x_n\rightarrow x$ when $\vartheta(x_n)\stackrel{\bar{\mathbb R}_+}{\rightarrow} \vartheta(x)$ and $\xi(x_n)\stackrel{\mathbb R^d}{\rightarrow} \xi(x)$. Let $\tau_0$ be the class of subsets of $\bar{\mathbb R}^d_0$ such that $A\in\tau_0$ if and only if for any $x\in A$ and any $x_1,x_2,\dots\in\bar{\mathbb R}^d_0$ with $x_n\rightarrow x$ there is an $N$ such that for all $n\ge N$, $x_n\in A$. It is straightforward to show that $\tau_0$ is a topology. In this topology compact sets are closed sets that are bounded away from $0$. The Borel $\sigma$-algebra on $\mathbb {\bar R}^d_0$ is the $\sigma$-algebra generated by  $\tau_0$; we denote it by $\mathfrak B(\mathbb {\bar R}^d_0)$.

To define convergence of a sequence in $\mathbb {\bar R}^d$, we first define convergence to a point $x\neq0$ as before. For $x_1,x_2,\dots\in\mathbb{\bar R}^d$ we write 
$x_n\rightarrow 0$ when $\vartheta(x_n)\stackrel{\bar{\mathbb R}_+}{\rightarrow} 0$. Note that  if $x,x_1,x_2,\dots\in\bar{\mathbb R}^d\setminus\mathbb I^{d-1}$ then $x_n\rightarrow x$ if and only if $x_n\stackrel{\mathbb R^d}{\rightarrow}x$. We define a topology $\tau$ on $\mathbb {\bar R}^d$ in a manor analogous to the previous case. Here the compact sets are the closed sets. The Borel $\sigma$-algebra on $\mathbb {\bar R}^d$ is the $\sigma$-algebra generated by $\tau$; we denote it by $\mathfrak B(\mathbb {\bar R}^d)$.

For notational convenience, throughout this paper we identify Borel measures on $\mathbb{\bar R}^d_0$ with Borel measures on $\mathbb {\bar R}^d$ that place no mass at zero. Likewise, we identify Borel measures on $\mathbb R^d$ with Borel measures on $\mathbb {\bar R}^d$ that place no mass on $\mathbb I^{d-1}$.

A Borel measure on $\bar{\mathbb R}^d_0$ is called a \textbf{Radon measure} if it is finite on any subset that is bounded away from $0$. Note that all L\'evy measures and all Rosi\'nski measures are Radon measures on $\mathbb {\bar R}^d_0$. We now define vague convergence on the spaces $\mathbb {\bar R}_0^d$, $\mathbb {\bar R}^d$, and $\mathbb R^d_0$. All three follow from the general definition of vague convergence on a topological space (see e.g.\ Chapter 3 in \cite{Resnick:1987}).

\begin{definition}\label{defn: vague conv on R0}
Let $\mu_0,\mu_1,\mu_2,\dots$ be Radon measures on $\bar{\mathbb R}^d_0$. We write $\mu_n\conv\mu_0$ on $\mathbb {\bar R}^d_0$ if for all continuous, real-valued functions $f$ on $\bar{\mathbb R}^d$ vanishing on a neighborhood of zero 
\begin{eqnarray}\label{eq: defn vague conv}
\lim_{n\rightarrow\infty}\int_{\bar{\mathbb R}^d} f(x)\mu_n(\rd x) = \int_{\bar{\mathbb R}^d} f(x)\mu_0(\rd x).
\end{eqnarray}
\end{definition}

\begin{definition}\label{defn: vague conv on R}
Let $\mu_0,\mu_1,\mu_2,\dots$ be finite Borel measures on $\bar{\mathbb R}^d$. We write $\mu_n\conv\mu_0$ on $\mathbb {\bar R}^d$ if for all continuous, real-valued functions $f$ on $\bar{\mathbb R}^d$ \eqref{eq: defn vague conv} holds.
\end{definition}

\begin{definition}
Let $\mu_0,\mu_1,\mu_2,\dots$ be Borel measures on $\mathbb R^d_0$ such that for all $0<a<b<\infty$ we have $\mu_n([a,b])<\infty$. If \eqref{eq: defn vague conv} holds for any continuous, real-valued function $f$ vanishing on a neighborhood of zero and on a neighborhood of infinity we write $\mu_n\conv\mu_0$ on $\mathbb R^d_0$.
\end{definition}

It is not difficult to show that $\bar{\mathbb R}^d_0$ with the topology $\tau_0$ and $\bar{\mathbb R}^d$ with the topology $\tau$ are locally compact Hausdorff spaces with a countable basis. This implies that a number of standard results about vague convergence on $\mathbb R^d$ extend to these spaces. In particular, Theorems 3.12 and 3.16 in \cite{Resnick:1987}, imply that a version of the Portmanteau Theorem and Helly's Selection Theorem hold. The latter can be formulated as follows.

\begin{proposition}\label{prop: convergent subseq} Let $\{\mu_n\}$ be a sequence of Borel measures on $\bar{\mathbb R}^d$ with
$
\sup\mu_n(\mathbb {\bar R}^d)<\infty.
$
There exists a subsequence $\{\mu_{n_k}\}$ and a finite Borel measure $\mu_0$ on $\mathbb {\bar R}^d$ such that $\mu_{n_k}\conv\mu_0$ on $\mathbb {\bar R}^d$.
\end{proposition}

We now give a useful characterization of vague convergence on $\bar{\mathbb R}^d$ for the special case when none of the measures place mass on $\mathbb I^{d-1}$. Let $C^b$ be the class of Borel functions mapping $\mathbb {\bar R}^d$ into $\mathbb R$, which are continuous and bounded on $\mathbb {\bar R}^d\setminus \mathbb I^{d-1}$.  We make no assumption about their behavior on $\mathbb I^{d-1}$.

\begin{lemma}\label{lemma: C-sharp suff for vague conv}
Let $\mu_0,\mu_1,\mu_2,\dots$ be finite Borel measures on $\mathbb{\bar R}^d$ such that $\mu_n(\mathbb I^{d-1})=0$ for all $n$. Then $\mu_n\conv \mu_0$ on $\mathbb{\bar R}^d$ if and only if $\int_{\mathbb {\bar R}^d} f(x)\mu_n(\rd x)\rightarrow\int_{\mathbb {\bar R}^d} f(x)\mu_0(\rd x)$ for all $f\in C^b$.
\end{lemma}

\begin{proof}
Assume that $\mu_n\conv \mu_0$ on $\mathbb{\bar R}^d$, let $H=\{T\in(0,\infty): \mu_0(|x|=T) = 0\}$, and fix $f\in C^b$.  This means that there is a $K$ such that $|f(x)|\le K$ for all $x\in\mathbb R^d$. Without loss of generality assume that $f(x)\ge0$. From the Portmanteau Theorem (Theorem 3.12 in \cite{Resnick:1987}) it follows that for all $T\in H$
$$
\lim_{n\rightarrow\infty} \int_{|x|\le T}f(x)\mu_n(\rd x) = \int_{|x|\le T}f(x)\mu_0(\rd x).
$$
Thus
\begin{eqnarray*}
\liminf_{n\rightarrow\infty}\int_{\mathbb {\bar R}^d} f(x)\mu_n(\rd x) &\ge& \lim_{H\ni T\uparrow\infty}\lim_{n\rightarrow\infty} \int_{|x|\le T}f(x)\mu_n(\rd x) \\
&=& \lim_{H\ni T\uparrow\infty}\int_{|x|\le T}f(x)\mu_0(\rd x)= \int_{\mathbb {\bar R}^d} f(x)\mu_0(\rd x),
\end{eqnarray*}
where the last equality follows by dominated convergence. Since $\mu_0(\mathbb I^{d-1})=0$, for any $\delta>0$ there is a $T_\delta\in H$ with $\mu_0(|x|\ge T_\delta)\le\delta/ K$. Thus
\begin{eqnarray*}
\int_{\mathbb {\bar R}^d}f(x)\mu_n(\rd x) &\le& \int_{|x|\le T_\delta}f(x)\mu_n(\rd x) + K\mu_n(|x|>T_\delta)\\
&\rightarrow& \int_{|x|\le T_\delta}f(x)\mu_0(\rd x) + K\mu_0(|x|>T_\delta)\\
&\le& \int_{\mathbb {\bar R}^d}f(x)\mu_0(\rd x) + \delta.
\end{eqnarray*}
Since this holds for all $\delta>0$, $\limsup_{n\rightarrow\infty}\int_{\mathbb {\bar R}^d}f(x)\mu_n(\rd x)\le \int_{\mathbb {\bar R}^d}f(x)\mu_0(\rd x)$, and hence 
$$
\lim_{n\rightarrow\infty}\int_{\mathbb {\bar R}^d}f(x)\mu_n(\rd x)=\int_{\mathbb {\bar R}^d}f(x)\mu_0(\rd x).
$$
The other direction follows from Definition \ref{defn: vague conv on R}.
\end{proof}

We conclude this section by recalling a standard result about convergence of infinitely divisible distributions in terms of vague convergence of their L\'evy measures. The following is a variant of Theorem 3.1.16 and Corollary 2.1.17 in \cite{Meerschaert:Scheffler:2001}. 

\begin{proposition} \label{prop: ID limits}
Let $\mu_n=ID(A_n,M_n,b_n)$. If $\mu_n\conw\mu$ then $\mu=ID(A,M,b)$. Moreover, $\mu_n\conw\mu$ if and only if $M_n \conv M$ on $\bar{\mathbb R}^d_0$, $b_n\rightarrow b$, and
\begin{eqnarray}\label{eq: gaus comp inf div}
\lim_{\epsilon\downarrow0} \lim_{n\rightarrow\infty} \left(A_n + \int_{|x|\le \epsilon} xx^T M_n(\rd x)\right) = A.
\end{eqnarray}
The result remains true if \eqref{eq: gaus comp inf div} is replaced by
\begin{eqnarray}\label{eq: gaus comp inf div limsup}
\lim_{\epsilon\downarrow0} \liminf_{n\rightarrow\infty} \left(A_n + \int_{|x|\le \epsilon} xx^T M_n(\rd x)\right) = \lim_{\epsilon\downarrow0} \limsup_{n\rightarrow\infty} \left(A_n + \int_{|x|\le \epsilon} xx^T M_n(\rd x)\right) = A.
\end{eqnarray}
\end{proposition}
In the above and throughout, convergence of matrices should be interpreted as pointwise convergence of the components.

\section{Extended Rosi\'nski Measure}\label{sec: extended Rosinski measure}

We now return to our discussion of the L\'evy measures of extended $p$-tempered $\alpha$-stable distributions. Recall that the L\'evy measure of such a distribution can be given by \eqref{eq: Levy measure of ets in terms of R}. In this section, we will put it into a form that is easier to work with. First, let $\nu$ be a Borel measure on $\mathbb{\bar R}^d$ 
such that if $B \in\mathfrak B(\bar{\mathbb R}^d)$ then if $\alpha\in(0,2)$ 
\begin{eqnarray}
\nu(B) = \int_{\mathbb R^d}1_B(x)\left(|x|^2\wedge|x|^\alpha\right) R(\rd x) + \int_{\mathbb S^{d-1}} 1_B(\infty x) Q_x(\{0\})\sigma(\rd x)
\end{eqnarray}
and
\begin{eqnarray}
\nu(B) = \left\{
\begin{array}{lr}
\int_{\mathbb R^d}1_B(x)\left(|x|^2\wedge[1+\log^+|x|]\right)R(\rd x) & \mbox{if}\ \alpha=0\\
\int_{\mathbb R^d}1_B(x)\left(|x|^2\wedge1\right)R(\rd x) & \mbox{if}\ \alpha<0
\end{array}\right..
\end{eqnarray}
Note that $\nu(\{0\})=0$ and by \eqref{eq: integ cond on R} $\nu$ is a finite measure. In particular, $\nu$ is a Radon measure on $\mathbb {\bar R}^d_0$.  We will call it the \textbf{extended Rosi\'nski measure}. From $\nu$ we get $R$ back by
\begin{eqnarray}\label{eq: back to nu from R}
R(\rd x) = \left\{
\begin{array}{ll}
\left(|x|^2\wedge|x|^\alpha\right)^{-1} \nu_{|_{\mathbb R^d}}(\rd x) & \mbox{if}\ \alpha\in(0,2)\\
\left(|x|^2\wedge[1+\log^+|x|]\right)^{-1} \nu_{|_{\mathbb R^d}}(\rd x) & \mbox{if}\ \alpha=0\\
\left(|x|^2\wedge1\right)^{-1} \nu_{|_{\mathbb R^d}}(\rd x) & \mbox{if}\ \alpha<0
\end{array}\right.,
\end{eqnarray}
where $\nu_{|_{\mathbb R^d}}$ is the restriction of $\nu$ to $\mathbb R^d$.

\begin{remark}
Let $\nu$ be any finite Borel measure on $\mathbb {\bar R}^d$ with $\nu(\{0\})=0$. For any $p>0$ and $\alpha\in(0,2)$, $\nu$ is the extended Rosi\'nski measure of some distribution in $\ets$. If, in addition, $\nu(\mathbb I^{d-1})=0$ then for any $p>0$ and $\alpha<2$, $\nu$ is the extended Rosi\'nski measure of some distribution in $\ets$.
\end{remark}

\begin{proposition}\label{prop: uniq determ}
For a fixed $\alpha<2$ and $p>0$, the extended Rosi\'nski measure $\nu$ is uniquely determined by the L\'evy measure of the extended $p$-tempered $\alpha$-stable distribution.
\end{proposition}

\begin{proof}
This follows from the fact that $\nu$ is uniquely determined by $R$ and $Q_u(\{0\})\sigma(\rd u)$. In \cite{Grabchak:2011} it was shown that $R$ is uniquely determined by the L\'evy measure of the $p$-tempered $\alpha$-stable part, and Remark 14.4 in \cite{Sato:1999} says that $Q_u(\{0\})\sigma(\rd u)$ is uniquely determined by the L\'evy measure of the $\alpha$-stable part.
\end{proof}

\begin{definition}
A distribution in $\ets$ with Gaussian part $A$, extended Rosi\'nski measure $\nu$, and shift $b$ is denoted by $\ets(A,\nu,b)$.
\end{definition}

We conclude this section by giving a representation of the L\'evy measure of a distribution in $\ets$ in terms of its extended Rosi\'nski measure. Fix $\mu\in \ets$ with L\'evy measure $M$ given by \eqref{eq: levy of ETS}, and let $f$ be any Borel function, which is integrable with respect to $M$. If $\alpha<0$ then
\begin{eqnarray}\label{eq: integ levy a < 0}
\int_{\mathbb R^d} f(x) M(\rd x) &=& \int_{\mathbb {\bar R}^d}\int_0^\infty f(tx) t^{-1-\alpha} e^{-t^p} \rd t\frac{1}{1\wedge|x|^2}\nu(\rd x),
\end{eqnarray}
if $\alpha=0$ then
\begin{eqnarray}\label{eq: integ levy a = 0}
\int_{\mathbb R^d} f(x) M(\rd x) &=& \int_{\mathbb {\bar R}^d}\int_0^\infty f(tx) t^{-1} e^{-t^p} \rd t \frac{1}{|x|^2\wedge[1+\log^+|x|]}\nu(\rd x),
\end{eqnarray}
and if $\alpha\in(0,2)$ then
\begin{eqnarray}
\int_{\mathbb R^d} f(x) M(\rd x) &=& \int_{\mathbb S^{d-1}} \int_0^\infty f(tx) t^{-1-\alpha}\rd t Q_x(\{0\})\sigma(\rd x)\nonumber\\
&& \ \ \ \ \  + \int_{\mathbb R^d}\int_0^\infty f(tx) t^{-1-\alpha} e^{-t^p} \rd t R(\rd x)\nonumber\\
&=& \int_{\mathbb I^{d-1}} \int_0^\infty f(t\xi(x)) t^{-1-\alpha}e^{-(t/|x|)^p}\rd t\nu(\rd x)\nonumber\\
&& \ \ \ \ \ + \int_{\mathbb R^d} \int_0^\infty f(t\xi(x)) t^{-1-\alpha} e^{-(t/|x|)^p} \rd t |x|^\alpha R(\rd x)\nonumber\\
&=& \int_{\mathbb{\bar R}^d}\int_0^\infty f(t\xi(x)) t^{-1-\alpha} \frac{e^{-(t/|x|)^p}}{1\wedge |x|^{2-\alpha}} \rd t \nu(\rd x).\label{eq: integ levy}
\end{eqnarray}

\section{Sequences of Extended Tempered Stable Distributions}\label{sec: Limits of Sequences of Extended Tempered Stable Distributions}

We can now state our main result.

\begin{theorem}\label{thrm: closure}
Fix $\alpha<2$, $p>0$, and let $\mu_n =
ETS^p_\alpha(A_n,\nu_n,b_n)$. If $\mu_n\conw\mu$ then $\mu= ETS^p_\alpha(A,\nu,b)$. Moreover, $\mu_n\conw\mu$ if and only if $\nu_n\conv\nu$ on $\bar{\mathbb R}^d_0$, $b_n\rightarrow b$, and
\begin{eqnarray}\label{eq: gaussian comp}
\lim_{\epsilon\downarrow0}\lim_{n\rightarrow\infty} \left(A_n+H_n^\epsilon\right) = A,
\end{eqnarray}
where
\begin{eqnarray}\label{eq: defn Hn}
H_n^\epsilon = \int_{\mathbb |x|<\sqrt\epsilon}\frac{xx^T}{|x|^2}\int_0^{\epsilon|x|^{-1}}  t^{1-\alpha}e^{-t^p} \rd t \nu_n(\rd x).
\end{eqnarray}
The result remains true if \eqref{eq: gaussian comp} is replaced by
\begin{eqnarray}\label{eq: gaussian comp limsup}
\lim_{\epsilon\downarrow0}\liminf_{n\rightarrow\infty} \left(A_n+H_n^\epsilon\right) =\lim_{\epsilon\downarrow0}\limsup_{n\rightarrow\infty} \left(A_n+H_n^\epsilon\right) = A.
\end{eqnarray}
\end{theorem}

\begin{remark}\label{remark: no Gauss}
The extended Rosi\'nski measure does not contribute to the Gaussian part if and only if
\begin{eqnarray}\label{eq: no Gaus}
\lim_{\epsilon\downarrow0}\limsup_{n\rightarrow\infty} \tr H_n^\epsilon= \lim_{\epsilon\downarrow0}\limsup_{n\rightarrow\infty}\int_{\mathbb |x|<\sqrt\epsilon} \int_0^{\epsilon|x|^{-1}}  t^{1-\alpha}e^{-t^p} \rd t \nu_n(\rd x) = 0.
\end{eqnarray}
Since for any $\epsilon\in(0,1)$
\begin{eqnarray*}
\int_{\mathbb |x|<\epsilon}\nu_n(\rd x)\int_0^1  t^{1-\alpha}e^{-t^p} \rd t \le \tr H_n^\epsilon
\le \int_{\mathbb |x|<\sqrt\epsilon}\nu_n(\rd x)\int_0^{\infty}  t^{1-\alpha}e^{-t^p} \rd t,
\end{eqnarray*}
\eqref{eq: no Gaus} holds if and only if
\begin{eqnarray}\label{eq: pure closure gauss alt}
\lim_{\epsilon\downarrow0}\limsup_{n\rightarrow\infty}\int_{|x|<\epsilon}\nu_n(\rd x)=0.
\end{eqnarray}
\end{remark}

\begin{remark}\label{remark: in R}
Let $R_n$ be the Rosi\'nski measure corresponding to $\nu_n$ and let $R$ be the Rosi\'nski measure corresponding to $\nu$. When $\alpha\le0$ the condition $\nu_n\conv\nu$ on $\bar{\mathbb R}^d_0$ is equivalent to the condition $R_n\conv R$ on $\mathbb R^d_0$ and when $\alpha=0$ 
\begin{eqnarray}
\lim_{N\rightarrow\infty}\int_{|x|>N}\log|x|R_n(\rd x)=0
\end{eqnarray}
or when $\alpha<0$
\begin{eqnarray}
\lim_{N\rightarrow\infty}R_n(|x|>N)=0.
\end{eqnarray}
When $\alpha\in(0,2)$ the limit does not have an $\alpha$-stable part if and only if
\begin{eqnarray}
\lim_{N\rightarrow\infty}\int_{|x|>N}|x|^\alpha R_n(\rd x)=0.
\end{eqnarray}
\end{remark}

To facilitate the proof of Theorem \ref{thrm: closure}, we begin with several lemmas.

\begin{lemma}\label{lemma: bound integ cos func}
Fix $\alpha<2$ and $p>0$. If $s\in\mathbb R$ with $|s|\le1$ then
\begin{eqnarray}
\int_0^\infty\left(\cos\left(ts\right)-1\right)t^{-1-\alpha}e^{-t^p}\rd t \le
-\frac{11}{24}s^2\int_0^1 t^{1-\alpha}e^{-t^p}\rd t.
\end{eqnarray}
\end{lemma}
\begin{proof}
We have
\begin{eqnarray*}
 &&\int_0^\infty\left(\cos \left(ts\right)-1\right)t^{-1-\alpha}e^{-t^p}\rd t \le \int_0^1 \left(\cos \left(ts\right)-1\right)t^{-1-\alpha}e^{-t^p}\rd t\\
&&\qquad \qquad \qquad\qquad\qquad\quad \le \int_0^1\left( \frac{s^4 t^4}{24}-\frac{s^2t^2}{2} \right)t^{-1-\alpha} e^{-t^p}\rd t \le -\frac{11}{24}s^2\int_0^1 t^{1-\alpha}e^{-t^p}\rd t,
\end{eqnarray*}
where the second line follows by the Taylor expansion of cosine and the remainder theorem for alternating series.
\end{proof}

\begin{lemma}\label{lemma: finite sup}
Let the sequence $\{\mu_n\}$ be as in Theorem \ref{thrm: closure}.\\
\noindent 1. If $\mu_n\conw\mu$ for some probability measure $\mu$ then $\sup\nu_n(\bar{\mathbb R}^d) < \infty$.\\
\noindent 2. If $\nu_n\conv\nu$ on $\bar{\mathbb R}^d_0$ for some finite measure $\nu$ then for any $\delta>0$, $\sup\nu_n(|x|\ge \delta) < \infty$.\\
\noindent 3. If \eqref{eq: gaussian comp limsup} holds with some matrix $A$
then for any $\delta>0$, $\sup\nu_n(|x|<\delta) < \infty$.
\end{lemma}
\begin{proof}
The second part follows immediately from Definition \ref{defn: vague conv on R0}. To show the first part, assume that $\mu_n\conw\mu$ and get $R_n$ from $\nu_n$ by \eqref{eq: back to nu from R}. Lemma \ref{lemma: bound integ cos func} implies that for $|z|\le1$
\begin{eqnarray*}
|\hat \mu_n(z)| &\le& \left|\exp\left\{\int_{\mathbb R^d}\int_0^\infty\left(e^{it\langle x,z\rangle}-1- i\frac{t\langle x,z\rangle}{1+|x|^2}\right) t^{-1-\alpha} e^{-t^p} \rd t R_n(\rd x)\right\}\right|\\
&\le& \exp\left\{\int_{|x|\le1}\int_0^\infty\left(\cos \left(t\langle x,z\rangle\right)-1\right) t^{-1-\alpha} e^{-t^p} \rd t R_n(\rd x)\right\}\\
&\le& \exp\left\{ -\frac{11}{24}\int_0^1 t^{1-\alpha}e^{-t^p}\rd t\int_{|x|\le1}\langle x,z\rangle^2R_n(\rd x)\right\},
\end{eqnarray*}
where the first inequality follows by the fact that we can write $\mu_n$ as the convolution of a Gaussian, an element of $\ts$, and (when $\alpha\in(0,2)$) an $\alpha$-stable distribution. By Proposition 2.5 in \cite{Sato:1999} $|\hat\mu_n(z)|\rightarrow|\hat\mu(z)|$ uniformly on compact sets, and for some $b>0$, $|\hat\mu(z)|> b$ on a neighborhood of zero. Thus, on this neighborhood, for large enough $n$,
\begin{eqnarray*}
b < \exp\left\{ -\frac{11}{24}\int_0^1 t^{1-\alpha}e^{-t^p}\rd t\int_{|x|\le1}\langle x,z\rangle^2R_n(\rd x)\right\},
\end{eqnarray*}
which implies that 
$\sup\int_{|x|\le1}\langle x,z\rangle^2R_n(\rd x)<\infty$ for every $z\in\mathbb R^d$, and hence 
$$ 
\sup\nu_n\left(|x|\le 1\right) < \infty.
$$

By Proposition \ref{prop: ID limits}, $\mu$ is infinitely divisible. Let $M_n$ be the L\'evy measure of $\mu_n$ and let $M$ be the L\'evy measure of $\mu$. Let $f_1$ be a non-negative, continuous, bounded, real-valued function vanishing on a neighborhood of zero such that $f_1(y) = 1$ for $|y|\ge1$. When $\alpha\in(0,2)$ by \eqref{eq: integ levy}
\begin{eqnarray*}
\int_{\mathbb R^d} f_1(x)M_n(\rd x) &=& \int_{\mathbb{\bar R}^d}\int_0^\infty f_1(\xi(x)t)t^{-1-\alpha}\frac{e^{-(t/|x|)^p}}{1\wedge |x|^{2-\alpha}}\rd t\nu_n(\rd x)\nonumber\\
&\ge& \int_{\mathbb I^{d-1}}\int_1^\infty t^{-1-\alpha}\rd t\nu_n(\rd x)\nonumber\\
&& \ \ \ +\int_{\infty>|x|\ge1}\int_1^2 t^{-1-\alpha}e^{-(t/|x|)^p}\rd t\nu_n(\rd x)\nonumber\\
&\ge& \alpha^{-1}\nu_n\left(\mathbb I^{d-1}\right)+e^{-2^p}\frac{2^\alpha-1}{\alpha2^\alpha}\nu_n({\infty>|x|\ge1})\nonumber\\
&\ge& e^{-2^p}\frac{2^\alpha-1}{\alpha2^\alpha}\nu_n(|x|\ge1).
\end{eqnarray*}
Similarly when $\alpha=0$ by \eqref{eq: integ levy a = 0}
\begin{eqnarray*}
\int_{\mathbb R^d}f_1(x)M_n(\rd x) &=& \int_{\mathbb R^d}\int_0^\infty f_1(xt)t^{-1}e^{-t^p}\rd t \frac{\nu_n(\rd x)}{|x|^2\wedge(1+\log^+|x|)}\\
&\ge& e^{-e^p}\int_{|x|\ge1}\int_{|x|^{-1}}^{e} t^{-1}\rd t \frac{\nu_n(\rd x)}{1+\log|x|} = e^{-e^p}\nu_n\left(|x|\ge1\right),
\end{eqnarray*}
and when $\alpha<0$ by \eqref{eq: integ levy a < 0}
\begin{eqnarray*}
\int_{\mathbb R^d}f_1(x)M_n(\rd x) &=& \int_{\mathbb R^d}\int_0^\infty f_1(xt)t^{-1-\alpha} e^{-t^p} \rd t \frac{1}{1\wedge|x|^2} \nu_n(\rd x)\\
&\ge& \nu_n(|x|\ge1) \int_1^\infty t^{-1-\alpha} e^{-t^p} \rd t.
\end{eqnarray*}
Proposition \ref{prop: ID limits} implies that the left side converges to $\int_{\mathbb R^d} f_1(x)M(\rd x)$ in all three cases. Thus, since $\int_{\mathbb R^d} f_1(x)M(\rd x)<\infty$, we have $\sup\nu_n\left(|x|\ge 1\right)<\infty$.

The third part follows from the fact that \eqref{eq: gaussian comp limsup} implies that for any $\epsilon>0$
\begin{eqnarray*}
\infty &>& \limsup_{n\rightarrow\infty} \tr H_n^\epsilon = \limsup_{n\rightarrow\infty}\int_{|x|<\sqrt\epsilon} \int_0^{\epsilon|x|^{-1}}t^{1-\alpha}e^{-t^p}\rd t\nu_n(\rd x)\\
&\ge& \limsup_{n\rightarrow\infty}\int_0^{\sqrt\epsilon}t^{1-\alpha}e^{-t^p}\rd t\int_{|x|<\sqrt\epsilon} \nu_n(\rd x),
\end{eqnarray*}
and hence $\sup_n \nu_n(|x|< \sqrt\epsilon)< \infty$.
\end{proof}

\begin{lemma}\label{lemma: Gaussian comp}
Let the sequence $\{\mu_n\}$ be as in Theorem \ref{thrm: closure} and let $M_n$ be the L\'evy measure of $\mu_n$. If $\sup\nu_n(\bar{\mathbb R}^d)<\infty$ then
\begin{eqnarray}
&& \lim_{\epsilon\downarrow0} \lim_{n\rightarrow\infty} \left( A_n+\int_{|x|\le\epsilon} xx^T M_n(\rd x)\right)\nonumber\\
&& \qquad\qquad  = \lim_{\epsilon\downarrow0} \lim_{n\rightarrow\infty} \left( A_n+\int_{\mathbb |x|<\sqrt\epsilon}\frac{xx^T}{|x|^2}\int_0^{\epsilon|x|^{-1}}  t^{1-\alpha}e^{-t^p} \rd t \nu_n(\rd x)\right),
\end{eqnarray}
whenever at least one of the limits exists. The result remains true if we replace $\lim_{n\rightarrow\infty}$ by $\liminf_{n\rightarrow\infty}$ or $\limsup_{n\rightarrow\infty}$.
\end{lemma}

\begin{proof}
We give the proof for the case when $\alpha\in(0,2)$ only. The other cases are similar. We can write
\begin{eqnarray}
\int_{|x|\le\epsilon} xx^T M_n(\rd x) &=& \int_{\mathbb I^{d-1}}\int_0^\epsilon \xi(x)[\xi(x)]^T t^{1-\alpha}\rd t\nu_n(\rd x)\nonumber\\
&& \ \ \ + \int_{\infty> |x|\ge1}\int_0^{\epsilon|x|^{-1}} xx^T t^{1-\alpha}e^{-t^p} \rd t|x|^{-\alpha} \nu_n(\rd x) \nonumber\\
&& \ \ \ + \int_{1> |x|\ge\sqrt\epsilon}\int_0^{\epsilon|x|^{-1}} xx^T t^{1-\alpha}e^{-t^p} \rd t |x|^{-2}\nu_n(\rd x) \nonumber\\
&& \ \ \ + \int_{\mathbb |x|<\sqrt\epsilon}\int_0^{\epsilon|x|^{-1}} xx^T t^{1-\alpha}e^{-t^p} \rd t|x|^{-2} \nu_n(\rd x)\nonumber\\
&=:& I_1^{n,\epsilon} + I_2^{n,\epsilon} + I_3^{n,\epsilon} + I_4^{n,\epsilon}.\nonumber
\end{eqnarray}
With $C:=\sup_n\nu_n(\bar{\mathbb R}^d)<\infty$ we have
\begin{eqnarray*}
\lim_{\epsilon\downarrow0}\limsup_{n\rightarrow\infty} \tr I_1^{n,\epsilon} &=& \lim_{\epsilon\downarrow0}\limsup_{n\rightarrow\infty}\nu_n(\mathbb I^{d-1}) \frac{\epsilon^{2-\alpha}}{2-\alpha}=0,\\
\lim_{\epsilon\downarrow0}\limsup_{n\rightarrow\infty}\tr I_2^{n,\epsilon}
&\le& \lim_{\epsilon\downarrow0}C\frac{\epsilon^{2-\alpha}}{2-\alpha}=0,
\end{eqnarray*}
and
\begin{eqnarray*}
\lim_{\epsilon\downarrow0}\limsup_{n\rightarrow\infty}\tr I_3^{n,\epsilon} &\le& \lim_{\epsilon\downarrow0}\limsup_{n\rightarrow\infty} \int_{1> |x|\ge\sqrt\epsilon} \int_0^{\sqrt\epsilon} t^{1-\alpha}\rd t \nu_n(\rd x)\\
&\le& \lim_{\epsilon\downarrow0} C\frac{\epsilon^{1-\alpha/2}}{2-\alpha} = 0.
\end{eqnarray*}
This completes the proof.
\end{proof}

\begin{proof}[Proof of Theorem \ref{thrm: closure}]
Let $M_n$ be the L\'evy measure of $\mu_n$.

Assume that $\mu_n\conw\mu$. By Proposition \ref{prop: ID limits} $\mu$ is infinitely divisible with some L\'evy triplet $(A,M,b)$ such that $b_n\rightarrow b$, $M_n\conv M$ on $\mathbb {\bar R}^d_0$, and \eqref{eq: gaus comp inf div} holds. Combining this with Lemmas \ref{lemma: finite sup} and \ref{lemma: Gaussian comp} gives \eqref{eq: gaussian comp} which implies \eqref{eq: gaussian comp limsup}. It remains to show that there is an extended Rosi\'nski measure $\nu$ such that $\mu=ETS^p_\alpha(A,\nu,b)$ and $\nu_n\conv \nu$ on $\mathbb {\bar R}^d_0$.

By Lemma \ref{lemma: finite sup}, $\sup \nu_n(\mathbb{\bar R}^d)<\infty$. Thus, Proposition \ref{prop: convergent subseq} implies that there is a finite Borel measure $\tilde\nu$ on $\mathbb {\bar R}^d$ and a subsequence $\{\nu_{n_j}\}$ such that  $\nu_{n_j}\conv \tilde\nu$ on $\mathbb {\bar R}^d$. Let $\nu$ be a finite Borel measure on $\mathbb {\bar R}^d$ such that $\nu_{|_{\bar{\mathbb R}^d_0}} = \tilde \nu_{|_{\bar{\mathbb R}^d_0}}$ and $\nu(\{0\})=0$. Clearly $\nu_{n_j}\conv \nu$ on $\mathbb {\bar R}^d_0$. Let $f$ be any continuous non-negative function on $\mathbb {\bar R}^d$ such that there are $\epsilon,K>0$ with $f(x)=0$ whenever $|x|\le\epsilon$ and $f(x)\le K$ for all $x\in\mathbb {\bar R}^d$. For $x\in\bar{\mathbb R}^d$ define
\begin{eqnarray}
g_\alpha(x)=\left\{\begin{array}{ll}
\int_{\epsilon}^{\infty} f(\xi(x)t) t^{-1-\alpha}\frac{e^{-(t/|x|)^p}}{1\wedge|x|^{2-\alpha}}\rd t & \alpha\in(0,2)\\
\int_{\epsilon|x|^{-1}}^\infty f\left(xt\right)t^{-1} \frac{e^{-t^p}}{|x|^2\wedge\left[1+\log^+|x|\right]}\rd t & \alpha=0\\
\int_{\epsilon|x|^{-1}}^\infty f(xt)t^{-1-\alpha} \frac{e^{-t^p}}{|x|^2\wedge1} \rd t & \alpha<0.
\end{array}\right..\label{eq: g alpha}
\end{eqnarray}
We will show that
\begin{eqnarray}\label{eq: conv galpha}
\lim_{j\rightarrow\infty}\int_{\mathbb{\bar R}^d}g_\alpha(x)\nu_{n_j}(\rd x)= \int_{\mathbb{\bar R}^d}g_\alpha(x) \tilde\nu(\rd x).
\end{eqnarray}
Assuming that this holds and observing that $g_\alpha(0)=0$ gives
\begin{eqnarray*}
\int_{\mathbb R^d} f(x) M(\rd x) &=& \lim_{j\rightarrow\infty}\int_{\mathbb R^d} f(x) M_{n_j}(\rd x) = \lim_{j\rightarrow\infty}\int_{\mathbb{\bar R}^d}g_\alpha(x)\nu_{n_j}(\rd x)\\
&=& \int_{\mathbb{\bar R}^d}g_\alpha(x)\tilde\nu(\rd x)
=\int_{\mathbb{\bar R}^d}g_\alpha(x)\nu(\rd x).
\end{eqnarray*}
This implies that $M$ is the L\'evy measure of a $p$-tempered $\alpha$-stable distribution with extended Rosi\'nski measure $\nu$. Since L\'evy measures are unique, this proves that the class $\ets$ is closed under weak convergence. Moreover, since, by Proposition \ref{prop: uniq determ}, $\nu$ is uniquely determined by $M$, $\nu_n\conv \nu$ on $\bar{\mathbb R}^d_0$.

We will now show that \eqref{eq: conv galpha} holds. By Definition \ref{defn: vague conv on R}, it suffices to show that $g_\alpha$ is bounded and continuous. When $\alpha\in(0,2)$ the facts that $\int_\epsilon^\infty t^{-1-\alpha}\rd t<\infty$ and that $f(\xi(x)t)\frac{e^{-(t/|x|)^p}}{1\wedge|x|^{2-\alpha}}$ is uniformly bounded show that $g_\alpha$ is bounded, and by dominated convergence it is continuous on $\mathbb{\bar R}^d$. When $\alpha<0$, $$
1_{[t>\epsilon|x|^{-1}]}f(xt)t^{-1-\alpha} \frac{e^{-t^p}}{|x|^2\wedge1} \le Ke^{-t^p}\left( t^{-1-\alpha}+ t^{1-\alpha}\epsilon^{-2} \right),
$$
which is integrable on $[0,\infty)$. Thus $g_\alpha$ is bounded, and by dominated convergence it is continuous on $\bar{\mathbb R}^d$. 

When $\alpha=0$, by Lemma \ref{lemma: C-sharp suff for vague conv}, it suffices to show that $g_\alpha$ is bounded and continuous only on $\mathbb R^d$. If $|x|\le1$ then
$$
1_{[t\ge \epsilon|x|^{-1}]} f(xt)t^{-1}e^{-t^p}|x|^{-2} \le 1_{[t\ge0]}K \epsilon^{-2}te^{-t^p},
$$
which is integrable with respect to $t$. If $|x|\ge1$ fix $\delta\in(0,|x|)$ and let $x'$ be such that $|x'-x|<\delta$. Then
\begin{eqnarray*}
1_{[t>\epsilon|x'|^{-1}]}f(x't)t^{-1}e^{-t^p}\left[1+\log|x'|\right]^{-1} &\le& 1_{[t\ge \epsilon|x'|^{-1}]}K t^{-1}e^{-t^p}\\
&\le& 1_{[t\ge \epsilon(|x|+\delta)^{-1}]}K t^{-1}e^{-t^p},
\end{eqnarray*}
which is integrable with respect to $t$. Thus, by dominated convergence $g_0$ is continuous on $\mathbb{R}^d$. To show that $g_0(x)$ is bounded, note that when $|x|\le1$ then, as before
\begin{eqnarray*}
g_0(x) \le K \epsilon^{-2}\int_0^\infty te^{-t^p}\rd t<\infty,
\end{eqnarray*}
and when $|x|>1$
\begin{eqnarray*}
g_0(x) &\le& K\left[1+\log|x|\right]^{-1} \int_{\epsilon|x|^{-1}}^{\epsilon e} t^{-1}\rd t + K \int_{\epsilon e}^\infty t^{-1}e^{-t^p} \rd t\\
&=& K + K\int_{\epsilon e}^\infty t^{-1}e^{-t^p} \rd t.
\end{eqnarray*}

Now for the other direction. Let $M$ be the L\'evy measure of $\mu$. Assume that $b_n\rightarrow b$,  \eqref{eq: gaussian comp limsup} holds, and $\nu_n\conv\nu$ on $\bar{\mathbb R}^d_0$. Lemma \ref{lemma: finite sup} implies that $\sup\nu_n(\mathbb{\bar R^d})<\infty$. Thus combining \eqref{eq: gaussian comp limsup} with Lemma \ref{lemma: Gaussian comp} gives \eqref{eq: gaus comp inf div}. To show that $M_n\conv M$ on $\mathbb {\bar R}^d_0$ we will show that every subsequence has a further subsequence that does this. Let $\{n_k\}$ be any increasing sequence in $\mathbb N$. By Proposition \ref{prop: convergent subseq} there is a subsequence $n_{k_j}$ and a finite Borel measure $\tilde\nu$ on $\mathbb {\bar R}^d$ such that $\nu_{n_{k_j}}\conv \tilde\nu$ on $\mathbb {\bar R}^d$. Clearly, $\nu_{|_{\mathbb {\bar R}^d_0}} = \tilde\nu_{|_{\mathbb {\bar R}^d_0}}$. Let $f$ be a continuous nonnegative function on $\mathbb{\bar R}^d$ satisfying the same assumptions as in the other direction, and define $g_\alpha$ by \eqref{eq: g alpha}. Observing that $g_\alpha(0)=0$ gives
\begin{eqnarray*}
\int_{\mathbb R^d}f(x)M_{n_{k_j}}(\rd x) &=& \int_{\mathbb {\bar R}^d} g_\alpha(x) \nu_{n_{k_j}}(\rd x)\\
&\rightarrow& \int_{\mathbb {\bar R}^d} g_\alpha(x) \tilde\nu(\rd x)=\int_{\mathbb {\bar R}^d} g_\alpha(x) \nu(\rd x)= \int_{\mathbb R^d}f(x)M(\rd x),
\end{eqnarray*}
where the convergence follows by arguments similar to the other direction.
\end{proof}

\section{Closure Properties}\label{sec: Closure}

In this section we will show that $\ets$ is, in fact, the smallest class that contains $\ts$ and is closed under weak convergence. In the following let $N(b,A)$ denote the Gaussian distribution with mean vector $b$ and covariance matrix $A$, and let $S_\alpha(\sigma,b)$ denote the $\alpha$-stable distribution with spectral measure $\sigma$ and shift $b$. This means that $N(b,A)=ID(A,0,b)$ and $S_\alpha(\sigma,b) = ID(0_{d\times d},K,b)$ where $K(B) = \int_{\mathbb S^{d-1}} \int_0^\infty 1_B(xt) t^{-1-\alpha} \rd t \sigma(\rd x)$ for $B\in\mathfrak B(\mathbb R^d)$.

\begin{proposition}\label{prop: nec of extended ts}
Fix $\alpha<2$ and $p>0$.\\
1. If $\mu = N(0,A)$ then there is a sequence $\{\mu_n\}$ in $\ts$ such that $\mu_n\conw \mu$.\\
2. If $\alpha\in(0,2)$ and $\mu= S_\alpha(\sigma,0)$ then there is a sequence $\{\mu_n\}$ in $\ts$ such that $\mu_n\conw \mu$.\\
3. The class $\ets$ is the smallest class that contains $TS^p_\alpha$ and is closed under weak convergence. Moreover, this class is closed under taking convolutions.
\end{proposition}

\begin{proof}
First observe that
$$
\lim_{s\rightarrow0} \frac{e^{i\langle x,z\rangle rs}-1-\frac{i\langle x,z\rangle sr}{1+|xr|^2s^2}}{s^2}=-\frac{1}{2}\langle x,z\rangle^2r^2.
$$
Let $R=N(0,cA)$, where $c = \left[\int_0^\infty r^{1-\alpha} e^{-r^p} \rd r\right]^{-1}$. Let $X=(X_1,\dots,X_d)^T\sim R$ and define
$$
R_n(B)=n^2\int_{\mathbb R^d}1_B(xn^{-1})R(\rd x), \qquad B\in\mathfrak B(\mathbb R^d).
$$
By \eqref{eq: integ cond on R}, this is the Rosi\'nski measure of some distribution in $\ts$. If $\mu_n = \ts(R_n,0)$ then
\begin{eqnarray}
C_{\mu_n}(z) &=& 
\int_{\mathbb R^d}\int_0^\infty\left( e^{i\langle x,z\rangle r} -1 -\frac{i\langle x,z\rangle r}{1+|x|^2r^2}\right)r^{-1-\alpha}e^{-r^p}\rd rR_n(\rd x)\nonumber\\
&=& n^2\int_{\mathbb R^d}\int_0^\infty\left( e^{i\langle x,z\rangle r/n} -1 -\frac{i\langle x,z\rangle r/n}{1+|x/n|^2r^2}\right) r^{-1-\alpha} e^{-r^p} \rd r R(\rd x)\nonumber\\
&\rightarrow& -\frac{1}{2}\int_{\mathbb R^d}\langle x,z\rangle^2R(\rd x) \int_0^\infty r^{1-\alpha} e^{-r^p} \rd r \nonumber\\
&=& -\frac{1}{2}\sum_{i=1}^d\sum_{j=1}^d z_i z_j E[X_iX_j]c^{-1}= -\frac{1}{2}\langle z,Az\rangle \nonumber,
\end{eqnarray}
where the third line follows by dominated convergence. For the second part let 
$$
R(B)=\int_{\mathbb S^{d-1}}\int_0^\infty 1_B(ut)e^{-t}t^{-\alpha}\rd t\sigma(\rd u), \qquad B\in\mathfrak B(\mathbb R^d),
$$
and note that
$$
\sigma(B) = \int_{\mathbb R^d} 1_B\left(\frac{x}{|x|}\right)|x|^\alpha R(\rd x), \qquad B\in\mathfrak B(\mathbb S^{d-1}).
$$
Let
$$
R_n(B)= n^{-\alpha} \int_{\mathbb R^d} 1_B(xn)R(\rd x), \qquad B\in\mathfrak B(\mathbb R^d).
$$
By \eqref{eq: integ cond on R}, this is the Rosi\'nski measure of some distribution in $\ts$. If $\mu_n= \ts(R_n,0)$ then
\begin{eqnarray}
C_{\mu_n}(z) &=& \int_{\mathbb R^d}\int_0^\infty\left( e^{i\langle x,z\rangle r} -1 -\frac{i\langle x,z\rangle r}{1+|x|^2r^2}\right) r^{-1-\alpha}e^{-r^p}\rd r R_n(\rd x) \nonumber\\
&=& n^{-\alpha}\int_{\mathbb R^d}\int_0^\infty\left( e^{i\langle x,z\rangle r n} -1 -\frac{i\langle x,z\rangle r n}{1+|xn|^2r^2}\right) r^{-1-\alpha} e^{-r^p} \rd r R(\rd x)\nonumber\\
&=& \int_{\mathbb R^d}\int_0^\infty\left( e^{i\langle x,z\rangle t/|x|} -1 -\frac{i\langle x,z\rangle t/|x|}{1+t^2}\right) t^{-1-\alpha} e^{-(t|x|^{-1}n^{-1})^p} \rd t |x|^\alpha R(\rd x)\nonumber\\
&\rightarrow& \int_{\mathbb R^d}\int_0^\infty\left( e^{i\langle x,z\rangle t/|x|}-1-\frac{i\langle x,z\rangle t/|x|}{1+t^2}\right) t^{-1-\alpha} \rd t |x|^\alpha R(\rd x)\nonumber\\
&=& \int_{\mathbb S^{d-1}}\int_0^\infty\left( e^{i\langle u,z\rangle t}-1-\frac{i\langle u,z\rangle t}{1+t^2}\right) t^{-1-\alpha} \rd t \sigma(\rd u)\nonumber,
\end{eqnarray}
where the third line follows by the substitution $t=rn|x|$ and the fourth by dominated convergence. The third part is an immediate consequence of the first two and Remark \ref{remark: closed under convolutions}.
\end{proof}

\begin{definition}\label{defn: elem ts vectors}
For $\alpha<2$ and $p>0$, a random vector is called an \textbf{elementary $\mathbf p$-tempered $\alpha$-stable random vector} on $\mathbb R^d$ if it can be written as $Ux$, where $x\in\mathbb R^d_0$ is a nonrandom vector and $U\sim ID(0, M,b)$ is an infinitely divisible random variable on $\mathbb R$ with $b\in\mathbb R$ and $M(\rd t)= c1_{[t>0]}t^{-1-\alpha}e^{-t^p}\rd t$, for some $c>0$.
\end{definition}

By \eqref{eq: inf div char func}, for $\lambda\in\mathbb R$, we have
\begin{eqnarray}
\mathrm Ee^{i\lambda U} &=& \exp\left\{ c\int_0^\infty \left( e^{i\lambda t} -1-\frac{i\lambda t}{1+t^2}\right)t^{-1-\alpha}e^{-t^p} \rd t + i\lambda b \right\}.
\end{eqnarray}
Thus for $z\in\mathbb R^d$
\begin{eqnarray*}
\mathrm Ee^{i\langle z,Ux\rangle} &=& \exp\left\{ c\int_0^\infty \left( e^{i\langle z,x\rangle t} -1-\frac{i\langle z,x\rangle t}{1+t^2} \right) t^{-1-\alpha}e^{-t^p} \rd t + i\langle z,xb\rangle \right\} \nonumber \\
&=& \exp\left\{ \int_{\mathbb R^d} \int_0^\infty \left( e^{i\langle y,z\rangle t} -1-\frac{i\langle y,z\rangle t}{1+t^2} \right)t^{-1-\alpha}e^{-t^p} \rd tR(\rd y) + i\langle z,xb\rangle \right\},
\end{eqnarray*}
where $R(\rd y)=c\delta_{x}(\rd y)$. Thus, a random vector is the finite sum of elementary $p$-tempered $\alpha$-stable random vectors if and only if its distribution is an element of $TS^p_\alpha$ with Rosi\'nski measure $R$ having a finite support.

\begin{theorem}\label{thrm: approx by elem}
Fix $\alpha<2$ and $p>0$. The class $ETS^p_\alpha$ is the smallest class of distributions closed under convolution and weak convergence and containing all elementary $p$-tempered
$\alpha$-stable distributions. In fact, $\mu\in ETS^p_\alpha$ if and only if there are probability measures $\mu_1,\mu_2,\dots$ with $\mu_n\conw\mu$ such that each $\mu_n$ is the distribution
of the sum of a finite number of independent elementary $p$-tempered $\alpha$-stable random vectors.
\end{theorem}

For the case when $p=1$ and $\alpha\in\{-1,0\}$ this was shown in Theorem F of \cite{Barndorff-Nielsen:Maejima:Sato:2006}. There the result followed from the properties of a certain integral representation. A similar representation for the case $\alpha<2$ and $p>0$ is given in \cite{Maejima:Nakahara:2009}.  However, in the case when $\alpha\in(0,2)$ the properties of the representation are different. Thus it appears that a proof analogous to that of \cite{Barndorff-Nielsen:Maejima:Sato:2006} can only be constructed when $\alpha\le0$. Instead, we base our proof on Theorem \ref{thrm: closure}.

\begin{proof}
In light of Proposition \ref{prop: nec of extended ts}, it suffices to show that we can approximate any distribution in $TS^p_\alpha$. Let $\mu=\ts(R,b)$ and let $\nu$ be its extended Rosi\'nski measure. Let  $(\nu_n)$ be any  sequence of finite measures on $\bar{\mathbb R}^d$ with a finite support such that $\nu_n(\{0\})=0$, $\nu_n(\mathbb I^{d-1})=0$, and $\nu_n\conv \nu$ on $\bar{\mathbb R}^d$ (such measures exist by e.g.\ Theorem 7.7.3 in \cite{Bauer:1981}). 
Let $\mu_n=\ets(0_{d\times d},\nu_n,b)$. Note that by the Portmanteau Theorem (Theorem 3.12 in \cite{Resnick:1987})
\begin{eqnarray*}
\lim_{\epsilon \downarrow 0}\limsup_{n\rightarrow\infty} \nu_n\left( |x| < \epsilon \right) \le \lim_{\epsilon \downarrow 0} \nu\left( |x| \le \epsilon \right)=0,
\end{eqnarray*}
where the final equality follows from the fact that $\nu$ is a finite measure that places no mass at $0$. Thus, $\mu_n\conw\mu$ by \eqref{eq: pure closure gauss alt} and Theorem \ref{thrm: closure}.
\end{proof}

Since all elementary $p$-tempered $\alpha$-stable distributions are proper, we immediately get the following.

\begin{corollary}
$ETS^p_\alpha$ is the smallest class of distributions closed under convolution and weak convergence and containing all proper $p$-tempered $\alpha$-stable distributions.
\end{corollary}

\section*{Acknowledgements}

The author wishes to thank his advisor Professor Gennady Samorodnitsky for many helpful discussions. This work was supported, in part, by funds provided by the University of North Carolina at Charlotte.

\bibliographystyle{plain}
\bibliography{GrabchakLimitsTS}
 \end{document}